\UseRawInputEncoding

\documentclass[final]{siamltex}
\usepackage[leqno]{amsmath}
\usepackage{amssymb}
\usepackage{framed}
\usepackage{makecell}
\usepackage{float}

\newtheorem{remark}{Remark}
\newtheorem{example}{Example}

\title{Solving separable convex optimization problems: Faster  prediction-correction framework\thanks{This work was supported by the National Natural Science Foundation of China under Grant 12171021 and
the Fundamental Research Funds for the Central Universities.}}

\author{Tao Zhang\footnotemark[2] \and Yong Xia\footnotemark[2] \and Shiru Li\footnotemark[2]}

\begin{document}
\maketitle
\renewcommand{\thefootnote}{\fnsymbol{footnote}}
\footnotetext[2]{School of Mathematical Sciences, Beihang University, Beijing, 100191, P. R. China  ({\tt (T. Zhang) shuxuekuangwu@buaa.edu.cn};  {\tt (Y. Xia) yxia@buaa.edu.cn}; {\tt (S. Li, Corresponding author) lishiru@buaa.edu.cn}).}

\begin{abstract}
He and Yuan's prediction-correction framework [SIAM J. Numer. Anal. 50: 700-709, 2012] is able to provide convergent algorithms for solving separable convex optimization problems at a rate of $O(1/t)$ ($t$ represents iteration times)  in both ergodic (the average of iteration) and pointwise senses.
This paper presents a faster prediction-correction framework at a rate of $O(1/t)$ in the non-ergodic sense (the last iteration) and $O(1/t^2)$ in the pointwise sense. 
Based the faster prediction-correction framework,   we give  three faster algorithms which enjoy $O(1/t)$ in the non-ergodic sense of primal-dual gap and  $O(1/t^2)$  in the pointwise sense. The first algorithm updates dual variable twice when solving two-block separable convex optimization with equality linear constraints. The second algorithm solves multi-block separable convex optimization problems with linear equality   constraints in Gauss-Seidel way. The third algorithm solves  minmax problems with larger step sizes.
\end{abstract}



\pagestyle{myheadings}
\thispagestyle{plain}
\markboth{T ZHANG,  Y XIA,  AND S LI}{A FASTER PREDICTION-CORRECTION FRAMEWORK}
\section{Introduction}
He and Yuan  presented a prediction-correction framework \cite{2012On} (or see \cite{2022on}) to analyze the convergence rate of the alternating direction method of multipliers (ADMM) \cite{gabay1976dual,glowinski1975}. Representing ADMM in their prediction-correction framework immediately leads to an $O(1/t)$  ($t$ represents the number of iterations)   convergence rate of the primal-dual gap in the ergodic sense\footnote{In this paper, for given iteration sequence $\{v^t\}$, the ergodic sense represents the average of $\{v^t\}$; the non-ergodic sense represents the last iteration sequence  of $\{v^t\}$; pointwise sense represents $\|v^t-v^{t-1}\|^2$.}
In a series of follow-up works, He et al. \cite{he2014strictly,doi:10.1137/21M1453463,he2016convergence,he2020optimal,2015On} further proved that the prediction-correction framework enjoys an $O(1/t)$ convergence rate in the pointwise sense.

The prediction-correction framework actually provides a unified approach for developing and analyzing algorithms that solve the following  two separable convex optimization problems.
\begin{example}\label{E1}
	The multi-block separable convex
	optimization problem with equality  constraints:
	\begin{equation}\label{P3}\tag{P1}
	\min\limits_{x_i}\left\{f(x)=\sum_{i=1}^{m}f_i(x_i):~ (Ax:=)\sum_{i=1}^{m}A_ix_i=b\right\},
	\end{equation}
	where $m\geq1$, $f_i:\mathbb{R}^{n_i}\rightarrow\mathbb{R}$ is a closed proper convex function,  $A_i\in\mathbb{R}^{l\times n_i}$, $\sum_{i=1}^{m}n_i=n$ and $b\in\mathbb{R}^{l}$. Define the Lagrangian function of \eqref{P3}:
	$$L(x,\lambda):=f(x)-\lambda^T(Ax-b),$$
	where $\lambda\in \mathbb{R}^{l}$ is the corresponding Lagrange multiplier.  
	For convenience, we define
	\begin{equation*}
	\begin{aligned}
	&
	\theta(u):=f(x)=\sum_{i=1}^{m}f_i(x_i),~
	u:=x=\begin{pmatrix}
	x_1\\ \vdots\\x_m
	\end{pmatrix},~
	w:=\begin{pmatrix}
	x_1\\ \vdots\\x_m\\\lambda
	\end{pmatrix},~F(w):=\begin{pmatrix}
	-A_1^T\lambda\\\vdots\\-A_m^T\lambda\\Ax-b
	\end{pmatrix}.
	\end{aligned}
	\end{equation*}
	Then according to the optimality condition,  \eqref{P3} is equivalent to find $w^*=(x_1^*,\dots,x_m^*,\lambda^*)$ such that \begin{equation}\label{P3'}\tag{P1'}
		0\in T(w^*),~{\rm where}~T(w):=\begin{pmatrix}
		\partial f_1(x_1)\\
		\vdots\\
		\partial f_m(x_m)\\
		0
		\end{pmatrix}+\begin{pmatrix}
		-A_1^T\lambda\\\vdots\\-A_m^T\lambda\\Ax-b
		\end{pmatrix}.
	\end{equation}
Noting that, for a given point $\hat{w}:=\begin{pmatrix}
\hat{x}\\\hat \lambda
\end{pmatrix}=\begin{pmatrix}
\hat{u}\\\hat \lambda
\end{pmatrix}$,  it holds that
$$\theta(\hat{u})-\theta(u)	-(w-\hat{w})^TF({w})=L(\hat x,\lambda)-L(x,\hat\lambda),$$
which is exactly the primal dual gap.\end{example}
\begin{example}\label{E2}
	The min-max  problem: \begin{equation}\label{P2}\tag{P2}
	\min\limits_{x} \max\limits_{y}
	\left\{\Phi(x,y):=f(x)-y^TAx-g(y)\right\},
	\end{equation}
	where $f:\mathbb{R}^{n}\rightarrow\mathbb{R}$ and  $g:\mathbb{R}^{m}\rightarrow\mathbb{R}$ are closed proper convex functions and  $A\in\mathbb{R}^{m\times n}$. 
For convenience, we define	
	\begin{equation*}
	\theta(u):=f(x)+g(y),~
	u=w:=\begin{pmatrix}
	x\\y
	\end{pmatrix},~F(w):=\begin{pmatrix}
	-A^Ty\\Ax
	\end{pmatrix}.
	\end{equation*}
		Then according to the optimality condition,  \eqref{P2} is equivalent to find $w^*=(x^*,y^*)$ such that\begin{equation}\label{P2'}\tag{P2'}
		0\in T(w^*),~{\rm where}~T(w):=\begin{pmatrix}
		\partial f(x)\\
		\partial g(y)
		\end{pmatrix}+\begin{pmatrix}
		-A^Ty\\Ax
		\end{pmatrix}.
		\end{equation}
 Noting that, for a given point $\hat{w}:=\hat{u}=\begin{pmatrix}
\hat{x}\\\hat y
\end{pmatrix}$, it holds that, 
$$\theta(\hat{u})-\theta(u)	-(w-\hat{w})^TF({w})=\Phi(\hat x,y)-\Phi(x,\hat y),$$
which is exactly the primal dual gap.
\end{example}

In this paper, we suppose  that the optimal solutions set of \eqref{P3} or \eqref{P2}
are nonempty and bounded.

Quite a few well-known algorithms that follow this framework include the augmented Lagrangian method (ALM) \cite{hestenes1969multiplier,powell1969method}, the proximal ALM \cite{rockafellar1976augmented,rockafellar1976monotone}, ADMM, the linear ADMM \cite{yang2013linearized}, and the strictly contractive Peaceman-Rachford splitting method \cite{he2014strictly,he2016convergence}. These algorithms are designed to solve two-block separable convex optimization problems with equality constraints. For solving multi-block  separable convex optimization problems with equality constraints, the Jacobian ALM \cite{doi:10.1137/130922793,xu2021parallel}, ADMM with substitution \cite{he2012alternating,he2017convergence}, and the ADMM-type algorithm \cite{he2021extensions}, are all based on the prediction-correction framework.
It is interesting to note that the analysis of the divergence of multi-block ADMM \cite{chen2016direct} is also based on the prediction-correction framework.
For solving min-max problems \eqref{P2},
Chambolle-Pock (CP) algorithm and their variants  \cite{chambolle2011first,chambolle2016ergodic,doi:10.1137/21M1453463,he2017algorithmic} could fall into the prediction-correction framework.

On the other hand, there are faster algorithms for solving unconstrained or simple-constrained convex optimization problems. Nesterov \cite{nesterov1983method} is the first to present an accelerated gradient method for unconstrained convex optimization with an $O(1/t^2)$  convergence rate. This method has been extended to composite convex optimization problems that involve the simple proximal operator \cite{beck2009fast,tseng2010} (see \cite{luo2022} for an explanation of second-order differential equations). Attouch \cite{doi:10.1137/20M1333316,2017Convergence} introduced a faster algorithm enjoying $O(1/t^2)$ in the pointwise sense for solving monotone inclusions problems by the continuous dynamical system approaches. 

Note that, by the prediction-correction framework, He and Yuan \cite{2012On} give $O(1/t)$ ergodic convergence rate of the primal-dual gap for the classical  ADMM. As for the non-ergodic convergence rate, only $O(1/\sqrt{t})$ non-ergodic convergence rate of classical ADMM is given in \cite[Section 3.4.5.1]{2020Accelerated}. This is slower  than the ergodic case.
Actually, as shown in \cite{Last2020}, 
Golowich et al. give the  theoretical guarantee that non-ergodic convergence is  slower than the ergodic sense when solving saddle point problems. Actually, the non-ergodic convergence is important in theory and also in practice.

For improving  the convergence rates in the non-ergodic case when solving linear constraints optimization problems only under convex assumption, 
Li and Lin \cite{li2019accelerated} and Tran-Dinh and Zhu \cite{2020NON} and  Valkonen \cite{valkonen2020inertial} present accelerated ADMM with $O(1/t)$ non-ergodic convergence rate when solving \eqref{P3} with $m=2$. The same non-ergodic rate result can also be obtained 
by the continuous dynamical system approaches \cite{2021A}.
By introducing the so called $  
nice~ primal$ $algorithmic$  $map$, Sabach and Teboulle \cite{sabach2022faster} present a class
of Lagrangian-based methods with $O(1/t)$ non-ergodic convergence rates of  both the function values and the feasibility measure when solving \eqref{P3} with $m=1$ or $2$. Note that all the papers \cite{li2019accelerated,2021A,sabach2022faster,2020NON,valkonen2020inertial}  do not give the $O(1/t^2)$  convergence rate in the pointwise sense. We can refer to Table \ref{ta1} for more details.

\begin{table}
	\caption{ Algorithms for solving \eqref{P3}}\label{ta1}
	\begin{center}\footnotesize
		\renewcommand{\arraystretch}{1.3}
		\begin{tabular}{|c|c|c|c|}\hline
			&
			\makecell{
				Problem}& 	
			\makecell{Rate of primal\\dual gap \\ function values or \\feasibility measure }	&\makecell{Rate in\\pointwise sense}\\		\hline
		{	ADMM}&$\eqref{P3}_{2}$&\makecell{$O(1/t)$(ergodic)\\
			$O(1\sqrt{t})$(non-ergodic)
			}&$O(1/t)$\\		\hline
			Li and Lin \cite{li2019accelerated}&$\eqref{P3}_{2}$&$O(1/{t})$(non-ergodic)&-\\	\hline
			\makecell
			{Tran-Dinh\\ and Zhu \cite{2020NON}}&$\eqref{P3}_{2}$&$O(1/{t})$(non-ergodic)&-\\ 	\hline
			Valkonen \cite{valkonen2020inertial}&$\eqref{P3}_{2}$&$O(1/{t})$(non-ergodic)&-\\	\hline
			\makecell{Sabach and \\Teboulle \cite{sabach2022faster}
			}&$\eqref{P3}_{2}$&$O(1/{t})$(non-ergodic)&-\\	\hline
			\makecell{Luo \cite{2021A}}&$\eqref{P3}_{2}$&$O(1/{t})$(non-ergodic)&-\\	\hline
			\makecell{Attouch et al. \cite{doi:10.1137/20M1333316,2017Convergence}}&$\eqref{P3}_{2}$&-&$O(1/t^2)$\\	\hline
			{\bf This paper}&\eqref{P3}&$O(1/{t})$(non-ergodic)&\makecell{$O(1/t^2)$}
			\\	\hline
		\end{tabular}
		\\$\eqref{P3}_{2}$ represents \eqref{P3} with $m=2$.
	\end{center}
\end{table}

Specially,  for the special case of \eqref{P3} with $m=1$ and equality constraints, Luo \cite{2021Accelerated,2021Ap}, He et al. \cite{2021Inertial} and 
Bo\c{t} et al. \cite{2021Fast}  give accelerated ALM with   $O(1/t^2)$ non-ergodic convergence of the function values and the feasibility measure. 
Only Bo\c{t} et al. \cite{2021Fast} give
the  convergence of the iterates under  smooth objective functions.
When solving the special case of \eqref{P3} with $m=2$ and equality constraints, 
 even if ADMM-type algorithms have easily solvable subproblems, the  subproblems of accelerated ALM  given in \cite{2021Inertial,2021Accelerated,2021Ap} may not be easily solvable. Without taking into account of the multi-block structure  is why the convergence rate is faster than $O(1/t)$. The special case of \eqref{P3} with $m=1$ is not the main work of this paper.
 
 Consider the problem \eqref{P2}. Chambolle and Pock \cite{chambolle2011first} give easily solvable subproblems algorithm (CP) with $O(1/t)$ ergodic convergence of primal-dual gap
 under $rs>\rho(A^TA)$\footnote{$\rho(A^TA)$ represents the spectrum of $A^TA$. $1/r$ and $1/s$ actually serve as the stepsizes in iteration algorithm.}.  Actually, CP can be seen as proximal point algorithm (PPA) for solving \eqref{P2}. Hence $O(1/t)$ in pointwise sense of iteration can be obtained. Jiang et al. \cite{2022so} and Li and Yan \cite{2022onrh} extend CP with $O(1/t)$ ergodic convergence under $rs>0.75\rho(A^TA)$ without pointwise convergence rate. Numerically, it indicates that small $rs$  will accelerate convergence when $\rho(A^TA)$ is large. Recently, He et al. \cite{doi:10.1137/21M1453463} give a CP  type algorithm for solving \eqref{P2} with both $O(1/t)$ in ergodic and pointwise sense under $rs>0.75\rho(A^TA)$.  Howere, we point out that the accelerated algorithms given in \cite{doi:10.1137/20M1333316,2017Convergence,li2019accelerated,sabach2022faster,2020NON,valkonen2020inertial} for solving \eqref{P2} is under the assumption $rs>\rho(A^TA)$. For more details, we can refer to Table \ref{ta2}.
\begin{table}
	\caption{Algorithms for solving \eqref{P2}}\label{ta2}
	\begin{center}\footnotesize
		\renewcommand{\arraystretch}{1.3}
		\begin{tabular}{|c|c|c|c|c|}\hline
			&
		Condition&\makecell{Rate of primal\\dual gap or\\ function values,\\feasibility measure }&	
\makecell{Rate in\\pointwise sense}\\		\hline
\makecell{Chambolle \\and Pock \cite{chambolle2011first}}&$rs>\rho(A^TA)$&$O(1/t)$(ergodic)&$O(1/t)$\\	\hline
\makecell
{Jiang et al. \cite{2022so}\\ Li and Yan \cite{2022onrh}}&$rs>0.75\rho(A^TA)$&$O(1/t)$(ergodic)&-\\	\hline
 He et al. \cite{doi:10.1137/21M1453463} &$rs>0.75\rho(A^TA)$&$O(1/t) $(ergodic)&$O(1/t)$\\	\hline
\makecell{Attouch et al.\\ \cite{doi:10.1137/20M1333316,2017Convergence}}&$rs>\rho(A^TA)$&-&$O(1/t^2)$\\	\hline
\cite{li2019accelerated,sabach2022faster,2020NON,valkonen2020inertial}&$rs>\rho(A^TA)$&$O(1/t)$(non-ergodic)&-\\	\hline
	{\bf This paper}&$rs>0.75\rho(A^TA)$&$O(1/t)$(non-ergodic)&\makecell{$O(1/t^2)$}\\	\hline
	
		\end{tabular}
	\end{center}
	
\end{table}

Considering that He and Yuan's prediction-correction framework only enjoys $O(1/t)$ convergence rate in ergodic sense and pointwise sense, and there is no unify framework with non-ergodic convergence rate when  solving \eqref{P3} and \eqref{P2}, it is necessary to establish a framework with faster convergence rate (non-ergodic sense and pointwise sense).

\noindent{\bf Contributions}  We list in the following the contributions of this paper:
\begin{itemize}
	\item[1.] We present a faster prediction-correction framework for solving \eqref{P3} and \eqref{P2}. Different from the algorithms which rely on $  
	nice~ primal$ $algorithmic$  $map$ in \cite{sabach2022faster} or the continuous dynamical system approaches in \cite{doi:10.1137/20M1333316,2017Convergence,2021A}, our ingredient is the prediction-correction framework in \cite{2012On} which enjoys $O(1/t)$ in ergodic sense of the primal-dual gap and $O(1/t)$ in  the pointwise sense. It is worth noting that our framework provides  algorithms at a rate of $O(1/t)$ in non-ergodic sense of primal-dual gap and $O(1/t^2)$   in the pointwise sense.
	\item[2.] He et al. \cite{he2014strictly,he2016convergence} propose a ADMM-type algorithm with dual variable updating twice for better numerical performance when solving \eqref{P3} with $m=2$ and equality constraints (or see Remark \ref{r1}).  Based on  the proposed  faster prediction-correction framework,
	we give a faster ADMM algorithm with dual variable updating twice for solving this problems  in Section \ref{s4.3}. This is the first paper that give algorithm updates dual variable twice with non-ergodic convergence.
	\item[3.]  Based on  the proposed  faster prediction-correction framework,
	we give a faster ADMM-type algorithm in  Gauss-Seidel way for solving \eqref{P3} in Section \ref{s4.1}.
	The existing accelerated algorithms with non-ergodic convergence rates solve  separable convex optimization problems with two-block equality constraints,  for example, \cite{li2019accelerated,2021A,sabach2022faster,2020NON,valkonen2020inertial}. Considering multi-block ADMM is divergent as shown in \cite{chen2016direct}. It seems that expanding  the accelerated algorithms in \cite{li2019accelerated,2021A,sabach2022faster,2020NON,valkonen2020inertial} to multi-block cases in Gauss-Seidel way is not easy. 
	 This is the first paper  to give  a faster  algorithm in Gauss-Seidel way with non-ergodic convergence rate for solving multi-block structure convex optimization problems with equality  constraints.  We can refer to  Table \ref{ta1} for detailed comparison for solving \eqref{P3}.
	\item[4.] Based on  the proposed  faster prediction-correction framework,
	we give a faster CP type algorithm for solving \eqref{P2} in Section \ref{s4.2}.  The faster CP type algorithm enjoys faster convergence such as $O(1/t)$ in non-ergodic sense, and $O(1/t^2)$ the pointwise sense compared to the non-accelerated  algorithms given in  \cite{doi:10.1137/21M1453463,2022so,2022onrh}. The convergence rate is established under the condition $rs>0.75\rho(A^TA)$ compared to the accelerated  algorithms \cite{doi:10.1137/20M1333316,2017Convergence,li2019accelerated,sabach2022faster,2020NON,valkonen2020inertial} with the condition $rs>\rho(A^TA)$. We can refer to  Table \ref{ta2} for detailed comparison for solving \eqref{P2}.
\end{itemize}



\section{He and Yuan's prediction-correction framework}
With $\theta(u)$, $u$, $w$, $F(w)$ and $T(w)$  defined in Example \ref{E1} or \ref{E2},
the fundmental algorithm for solving \eqref{P3'} and \eqref{P2'} (or \eqref{P3} and \eqref{P2}) is the proximal point algorithm (PPA), which was originally introduced by Martinet \cite{martinet1970regularisation}, reads as:
\begin{framed}
	\noindent {\bf PPA.} With a given $w^k$, find ${w}^{k+1}$ such that
	\begin{equation} 
	0\in T(w^{k+1})+Q(w^{k+1}-w^k),
	\end{equation}
	where $Q$ is  a symmetric defined matrix.
\end{framed}

For  general matrix $Q$, the subproblem of PPA may not be easily solved. Hence, an additional algorithm to solve the  subproblem of PPA is necessary.  Consider the separable structure of \eqref{P3'} and \eqref{P2'}, we can solve each block of  \eqref{P3'} or \eqref{P2'} in Jacobi or Gauss-Seidel way in order to reduce computation cost in every iteration (such as ADMM). 

Thus, we consider the following prediction-correction type PPA with special selection of the scaled matrixs.  Noting  that, by special selection of the scaled matrixs,  many famous algorithms fall into the following  prediction-correction framework, including the augmented Lagrangian method (ALM) \cite{hestenes1969multiplier,powell1969method}, the proximal ALM \cite{rockafellar1976augmented,rockafellar1976monotone}, ADMM, the linear ADMM \cite{yang2013linearized}, and the strictly contractive Peaceman-Rachford splitting method \cite{he2014strictly,he2016convergence}, the Jacobian ALM \cite{doi:10.1137/130922793,xu2021parallel}, ADMM with substitution \cite{he2012alternating,he2017convergence} for solving multi-block of \eqref{P3}, and
Chambolle-Pock (CP) algorithm and their variants  \cite{chambolle2011first,chambolle2016ergodic,doi:10.1137/21M1453463,he2017algorithmic}.  
\begin{framed}
	\noindent {\bf[Prediction step.]} With a given $w^k$, find $\widetilde{w}^k$ such that
	\begin{equation} \label{V3}\tag{PS}
	0\in T(\widetilde w^k)+L^TQL(\widetilde w^k-w^k),
	\end{equation}
	where $Q^T +Q\succ 0$ (noting that $Q$ is not necessarily symmetric), $L$ is a matrix.
	
	\noindent {\bf [Correction step.]} Update $v^{k+1}=Lw^{k+1}$ by
	\begin{equation}\label{V4}\tag{CS}
	v^{k+1}=v^k-M(v^k-\widetilde{v}^k), 
	\end{equation}where  $\widetilde v^k=L\widetilde w^k$.
\end{framed}
\begin{remark}
	Three algorithms satisfying \eqref{V3}-\eqref{V4} are given in  Section \ref{s2.2}-\ref{s2.4}. The definitions of $L$, $Q$ and $M$ for different algorithms are also given. 
\end{remark}

Then 
there exists  $\widetilde g^k\in (T-F)(\widetilde w^k)$ such  that 
\begin{equation}
\begin{aligned}
&(w-\widetilde w^k)^T(\widetilde g^k+F(\widetilde w^k)+L^TQL(\widetilde w^k-w^k)=0,~\forall w
\\\Longrightarrow&
\theta(u)-\theta(\widetilde{u}^k)+(w-\widetilde{w}^k)^TF(\widetilde{w}^k)\geq(w-\widetilde{w}^k)^TL^TQL(w^k-\widetilde{w}^k),~\forall w,
\end{aligned}
\end{equation}where the  inequality using convexity of $\theta$. Hence prediction-correction framework \eqref{V3}-\eqref{V4} infers
the following framework (in variational form) due to He and Yuan \cite{2012On}. It is fundamental in providing convergent algorithms for solving \eqref{P3} and \eqref{P2}. We can also refer to \cite{2022on} for a detailed understanding.
\begin{framed}
	\noindent {\bf[Prediction step.]} With a given $w^k$, find $\widetilde{w}^k$ such that
	\begin{equation}\label{G44} \theta(u)-\theta(\widetilde{u}^k)+(w-\widetilde{w}^k)^TF(\widetilde{w}^k)\geq(w-\widetilde{w}^k)^TL^TQL(w^k-\widetilde{w}^k), ~\forall w,
	\end{equation}
	where $Q^T +Q\succ 0$ (noting that $Q$ is not necessarily symmetric) and $L$ is a matrix.
	
	\noindent {\bf [Correction step.]} Update $v^{k+1}=Lw^{k+1}$ by
	\begin{equation}\label{G45}
	v^{k+1}=v^k-M(v^k-\widetilde{v}^k), 
	\end{equation}where  $\widetilde v^k=L\widetilde w^k$.
\end{framed}

To ensure the convergence of algorithms satisfying the above framework, He and Yuan \cite{2022on,2012On} add some assumptions on the selection of the matrices $Q$ and $M$. The following is such a commonly used condition.
\begin{framed}
	\noindent{\bf[Convergence Condition.]}
	For the given matrice $Q$ and nonsingular matrice $M$, setting
	\begin{equation}
	H:=	QM^{-1}\succ 0, 	\label{V5}\tag{CC1}
	\end{equation}
	\begin{equation}
	G:=Q^T +Q-M^THM\succ 0	.\label{V6}\tag{CC2}
	\end{equation}
\end{framed}


The convergence rates of prediction-correction framework \eqref{G44}-\eqref{G45} are established as follows.
\begin{theorem}[$O(1/t)$ ergodic convergence rate]\label{HYL1} {\rm(\cite{2012On,2022on})}\label{T1}
	Let $\{\widetilde{w}^k\}$ be generated by prediction-correction framework  \eqref{G44}-\eqref{G45} under \eqref{V5}-\eqref{V6}. Then we have $t=1,2,\dots,$
	\begin{equation*}
	\theta(\bar{u}^t)-\theta(u) +(\bar{w}^t-w)^TF({w})\leq\frac{1}{2(t+1)}\|v^0-Lw\|^2_{H},~\forall w,
	\end{equation*}
	where $
	\bar{w}^t=\frac{1}{t+1}\sum_{k=0}^{t}\widetilde{w}^k.
	$	
\end{theorem}

\begin{theorem}[$O(1/t)$  in the pointwise sense] {\rm(\cite{he2014strictly,doi:10.1137/21M1453463,he2016convergence,he2020optimal,2015On,2022on})}\label{T2}
	Let $\{\widetilde{w}^k\}$ be generated by prediction-correction framework  \eqref{G44}-\eqref{G45} under \eqref{V5}-\eqref{V6}. Then we have
	\begin{equation*}
	\|M(v^t-\widetilde{v}^t)\|^2_H\leq O(1/t),~t=1,2,\dots.
	\end{equation*}
\end{theorem}

In the following, we give some algorithms that satisfying \eqref{V3}-\eqref{V4} for solving \eqref{P3} and \eqref{P2} for better understanding of   prediction-correction framework \eqref{V3}-\eqref{V4}.
\subsection{{Algorithm satisfying \eqref{V3}-\eqref{V4} for solving \eqref{P3} with $m=2$}}\label{s2.2}
Consider  \eqref{P3} with $m=2$,  i.e.,
\begin{equation}\label{G23}
\min\limits_{x_1,x_2}\left\{f_1(x_1)+f_2(x_2):~ A_1x_1+A_2x_2=b\right\}.
\end{equation}

The strictly contractive Peaceman-Rachford splitting method \cite{he2014strictly,he2016convergence} for solving \eqref{G23} is given by:
\begin{equation}\label{G24}
\begin{cases}
x_1^{k+1}\in\arg\min\limits_{x_1}\left\{
f_1(x_1)-x_1^TA_1^T\lambda^k+\frac{\beta}{2}\|A_1x_1+A_2x_2^k-b\|^2+\frac{1}{2}\|x_1-x_1^k\|^2_P
\right\},\\
\lambda^{k+\frac{1}{2}}=\lambda^k-r\beta(A_1x_1^{k+1}+A_2x_2^k-b),\\
x_2^{k+1}\in\arg\min\limits_{x_2}\left\{
f_2(x_2)-x_2^TA_2^T\lambda^{k+\frac{1}{2}}+\frac{\beta}{2}\|A_1x_1^{k+1}+A_2x_2-b\|^2
\right\},	\\
\lambda^{k+1}=\lambda^{k+\frac{1}{2}}-s\beta(A_1x_1^{k+1}+A_2x_2^{k+1}-b),
~P\succ 0,~\beta>0.
\end{cases}
\end{equation}

\begin{remark}\label{r1}
		Algorithm \eqref{G24} reduces to ADMM when $r=0$ and $P=0$. Different from  ADMM, Algorithm \eqref{G24} updates $\lambda$ twice and convergence is established under weak conditions in \eqref{G53}.  As shown in \cite{he2014strictly,he2016convergence}, Algorithm \eqref{G24} enjoys  better numerical performance compared to ADMM.  Algorithm \eqref{G24} with  $r=s=1$ and $\alpha=0$ reduces to Peaceman-Rachford splitting method  (PRSM) for solving the dual of \eqref{G23}. Howere, the convergence of PRSM is established under the  strongly convex assumption. In order to remove the strongly convex assumption, 	He et al. \cite{he2014strictly,he2016convergence}  propose  some conditions on $r$ and $s$ to ensure convergence (or see \eqref{G53}). 
		
		We can set $P=\alpha I_{n_1}-\beta A_1^TA_1$ with $\alpha>\|A_1\|^2$ in order to enjoy easily solved subproblems.
\end{remark} 

For convenience, we define
	\begin{equation}\label{G25}
	\begin{aligned}
	&{L}=\begin{pmatrix}
	I_{n_1}&0&0\\
	0&I_{n_2}&0\\
	0&0&I_l
	\end{pmatrix},~
	Q=\begin{pmatrix}P&0&0\\
	0&\beta A_2^TA_2&-rA_2^T\\
	0&-A_2&\frac{1}{\beta}I_l
	\end{pmatrix},~M=\begin{pmatrix}
	I_{n_1}&0&0\\
	0&I_{n_2}&0\\
	0&-s\beta A_2&(r+s)I_l
	\end{pmatrix}.
	\end{aligned}
	\end{equation}
\begin{theorem}\label{t2.1}
For $L,~Q,~M$ defined in \eqref{G25} and $\theta(u), ~u,~w,~F(w),~T(w)$ defined in Example \ref{E1} with $m=2$, it holds that:
	\begin{itemize}
	\item[\textnormal{({1})}]     Algorithm \eqref{G24} satisfies  \eqref{V3}-\eqref{V4}.
	\item[\textnormal{({2})}]   $H$ and $G$ satisfying \eqref{V5}-\eqref{V6} are positive definite if $A_2$ is  full  column rank,   \begin{equation}\label{G53}
	r\in(-1,1),~s\in(0,1)~{\rm and}~r+s>0.
	\end{equation}
	\end{itemize}
\end{theorem}
 \begin{proof}
 	Setting
 	\begin{equation*}
 	\begin{aligned}
 	\widetilde{x}_1^k:=x_1^{k+1},~ \widetilde{x}_2^k:=x_2^{k+1}~{\rm and}~\widetilde{\lambda}^k:=\lambda^k-\beta(A_1\widetilde{x}_1^{k}+A_2x_2^k-b)
 	\end{aligned}
 	\end{equation*} in Algorithm \eqref{G24}.
 	
\noindent Proof of \textnormal{({1})}.
 The optimality condition of the $x_1$-subproblem reads as:
\begin{equation*}\begin{aligned}
0&\in \partial f_1(\widetilde x_1^k)-A_1^T{[\lambda^k-\beta(A_1\widetilde{x}_1^{k}+A_2x_2^k-b)]}+\alpha(\widetilde{x}_1^k-x_1^k)\\&=
 \partial f_1(\widetilde x_1^k)-A_1^T{\widetilde{\lambda}^k}+P(\widetilde{x}_1^k-x_1^k).
\end{aligned}
\end{equation*}
 The optimality condition of the $x_2$-subproblem reads as:
\begin{equation*}\begin{aligned}
0&\in \partial f_2(\widetilde x_2^k)-A_2^T[
\lambda^k-r\beta(A_1\widetilde{x}_1^{k}+A_2x_2^k-b)
]+\beta A_2^T[ A_1\widetilde{x}_1^{k}+A_2\widetilde{x}_2^k-b]\\&=
\partial f_2(\widetilde x_2^k)-A_2^T\widetilde{\lambda}^k+\beta A_2^T A_2(\widetilde{x}_2^k-x_2^k)-rA_2^T(\widetilde{\lambda}^k-\lambda^k).
\end{aligned}
\end{equation*}
 The definition of $\widetilde{\lambda}^k$ can be rewritten as:
\begin{equation*}
0=	A_1\widetilde{x}_1^{k}+A_2\widetilde{x}_2^k-b-A_2(\widetilde{x}_2^k-x_2^k)+\frac{1}{\beta}(\widetilde{\lambda}^k-\lambda^k).
\end{equation*}
Combining the above three relations, we obtain
	\begin{equation*}
	\begin{aligned}
	&0\in \begin{pmatrix}
	\partial f_1(\widetilde{x}_1^k)\\\partial f_2(\widetilde{x}_2^k)\\0
	\end{pmatrix}+
	\begin{pmatrix}
	-A_1^T\widetilde{\lambda}^k\\-A_2^T\widetilde{\lambda}^k\\A_1\widetilde{x}_1^{k}+A_2\widetilde{x}_2^k-b
	\end{pmatrix}
	+\begin{pmatrix}
	P(\widetilde{x}_1^k-x_1^k)\\
	\beta A_2^T A_2(\widetilde{x}_2^k-x_2^k)-rA_2^T(\widetilde{\lambda}^k-\lambda^k) \\
	-A_2(\widetilde{x}_2^k-x_2^k)+\frac{1}{\beta}(\widetilde{\lambda}^k-\lambda^k)
	\end{pmatrix},
	\end{aligned}
	\end{equation*}
	which is equivalent to$$
	\begin{aligned}
	&	0\in \begin{pmatrix}
\partial f_1(\widetilde{x}_1^k)\\\partial f_2(\widetilde{x}_2^k)\\0
	\end{pmatrix}+
	\begin{pmatrix}
	-A_1^T\widetilde{\lambda}^k\\-A_2^T\widetilde{\lambda}^k\\A_1\widetilde{x}_1^{k}+A_2\widetilde{x}_2^k-b
	\end{pmatrix}+{\begin{pmatrix}P&0&0\\
		0&\beta A_2^TA_2&-rA_2^T\\
		0&-A_2&\frac{1}{\beta}I_l
		\end{pmatrix}}_{}{\begin{pmatrix}
		x_1^k-\widetilde{x}_1^k\\
		x_2^k-\widetilde{x}_2^k\\
		\lambda^k-\widetilde{\lambda}^k
		\end{pmatrix}}.
	\end{aligned}  	
	$$Then  prediction step holds. It holds that 
\begin{equation}
\begin{aligned}
\lambda^{k+1}=&\lambda^{k+\frac{1}{2}}-s\beta(A_1\widetilde x_1^{k}+A_2\widetilde x_2^{k}-b)\\=&\lambda^k-r\beta(A_1\widetilde{x}_1^{k}+A_2x_2^k-b)-s\beta(A_1\widetilde x_1^{k}+A_2 x_2^{k}-b)+s\beta A_2(x_2^k-\widetilde x_2^{k})\\=&
\lambda^{k}-r(\lambda^k-\widetilde \lambda^k)-s(\lambda^k-\widetilde \lambda^k)+s\beta A_2(x_2^k-\widetilde x_2^{k}).
\end{aligned}
\end{equation}Thus, together with 	$\widetilde{x}_1^k=x_1^{k+1}$ and  $\widetilde{x}_2^k=x_2^{k+1}$, we have
$$\begin{pmatrix}
	x_1^{k+1}\\x_2^{k+1}\\\lambda^{k+1}
\end{pmatrix}=
\begin{pmatrix}
x_1^{k}\\x_2^{k}\\\lambda^{k}\end{pmatrix}-\begin{pmatrix}
I_{n_1}&0&0\\
0&I_{n_2}&0\\
0&-s\beta A_2&(r+s)I_l
\end{pmatrix}
\begin{pmatrix}
x_1^k-\widetilde x_1^{k}\\x_2^k-\widetilde x_2^{k}\\\lambda^k-\widetilde \lambda^{k}
\end{pmatrix}.
$$
The correction step holds.

\noindent Proof of \textnormal{({2})}.  For the positive definiteness of $H$ and $G$ satisfying \eqref{V5}-\eqref{V6}, we can refer to  \cite[Lemma 4.1]{he2016convergence} and \cite[Page 1480]{he2016convergence} for a similar proof. 
\end{proof}

\subsection{{Algorithm satisfying \eqref{V3}-\eqref{V4} for solving \eqref{P3}}}\label{s2.3}
The following algorithm for solving \eqref{P3} was first presented by  He et al. \cite{he2021extensions}.
\begin{framed}
	\noindent{\bf[Prediction step.]} With given $\beta>0$ and $(x_1^{k}, x_2^{k},\dots, x_m^{k},\lambda^k)$, find $(
	\widetilde{x}_1^k,\dots,\widetilde{x}_m^k,\widetilde \lambda^k)$ by
	\begin{footnotesize}
		\begin{equation}\label{G36}
		\begin{cases}
		\widetilde{x}_1^k=\arg\min\limits_{x_1}\left\{
		f_1(x_1)-x_1^TA_1^T\lambda^k+\frac{\beta}{2}\|A_1(x_1-x_1^k)\|^2
		\right\},\\
		\widetilde{x}_2^k=\arg\min\limits_{x_2}\left\{
		f_2(x_2)-x_2^TA_2^T
		\lambda^k+\frac{\beta}{2}\|A_1(\widetilde{x}_1^{k}-x_1^k)+A_2(x_2-x_2^k)\|^2
		\right\},\\
		~~~~~~~~~\vdots\\
		\widetilde{x}_i^k=\arg\min\limits_{x_i}\left\{
		f_i(x_i)-x_i^TA_2^T
		\lambda^k+\frac{\beta}{2}\|\sum_{j=1}^{i-1}A_j(\widetilde{x}_j^{k}-x_j^k)+A_i(x_i-x_i^k)\|^2
		\right\},\\
		~~~~~~~~~\vdots\\
		\widetilde{x}_m^k=\arg \min\limits_{x_m}\left\{
		f_m(x_m)-x_m^TA_2^T
		\lambda^k+\frac{\beta}{2}\|\sum_{j=1}^{m-1}A_j(\widetilde{x}_j^{k}-x_j^k)+A_m(x_m-x_m^k)\|^2
		\right\}\\
		\widetilde{\lambda}^k=\arg\max\limits_{\lambda}\left\{
		-\lambda^T(\sum_{j=1}^{m}A_j\widetilde{x}_j^k-b)-\frac{1}{2\beta}\|\lambda-\lambda^k\|^2
		\right\}.
		\end{cases}
		\end{equation}
	\end{footnotesize}
	\noindent{\bf [Correction step.]}Update  $(A_1x_1^{k+1},A_2x_2^{k+1},\dots,A_mx_m^{k+1},\lambda^{k+1})$ by
	\begin{footnotesize}
		\begin{equation}\label{G37}
		\begin{pmatrix}
		\sqrt{\beta}A_1x_1^{k+1}\\	\sqrt{\beta}A_2x_2^{k+1}\\\vdots\\	\sqrt{\beta}A_mx_m^{k+1}\\\frac{1}{	\sqrt{\beta}}\lambda^{k+1}
		\end{pmatrix}=\begin{pmatrix}
		\sqrt{\beta}A_1x_1^{k}\\	\sqrt{\beta}A_2x_2^{k}\\\vdots\\	\sqrt{\beta}A_mx_m^{k}\\\frac{1}{	\sqrt{\beta}}\lambda^{k}
		\end{pmatrix}-\begin{pmatrix}
		\alpha I_l&-\alpha I_l&0&\dots&0\\
		0&\alpha I_l&\ddots&\ddots&\vdots\\
		\vdots&\ddots&\ddots&-\alpha I_l&0\\
		0&\dots&0&\alpha I_l&0\\
		-\alpha I_l&0&\dots&0&I_l
		\end{pmatrix}\begin{pmatrix}
		\sqrt{\beta}( 	A_1x_1^{k}-A_1\widetilde{x}_1^k)\\	\sqrt{\beta}(A_2x_2^{k}-A_2\widetilde{x}_2^k)\\\vdots\\	\sqrt{\beta}(A_mx_m^{k}-A_m\widetilde{x}_m^k)\\\frac{1}{	\sqrt{\beta}}(\lambda^{k}-\widetilde{\lambda}^k)
		\end{pmatrix}.
		\end{equation}
	\end{footnotesize}
\end{framed}

For convenience, we define
\begin{small}
	\begin{equation}\label{G38}
	\begin{aligned}
	&
	L =\text{Diag}
	\begin{pmatrix}
	\sqrt{\beta} A_1, &
	\sqrt{\beta}A_2,&
	\dots,
	\sqrt{\beta}A_m,&
	\frac{1}{	\sqrt{\beta}}I_l
	\end{pmatrix},
	\\&
	Q=\begin{pmatrix}
	I_l&0&\dots&0&I_l\\
	I_l&I_l&\ddots&\vdots&I_l\\
	\vdots&\vdots&\ddots&0&\vdots\\
	I_l&I_l&\dots&I_l&I_l\\
	0&0&\dots&0&I_l
	\end{pmatrix},~M=\begin{pmatrix}
	\alpha I_l&-\alpha I_l&0&\dots&0\\
	0&\alpha I_l&\ddots&\ddots&\vdots\\
	\vdots&\ddots&\ddots&-\alpha I_l&0\\
	0&\dots&0&\alpha I_l&0\\
	-\alpha I_l&0&\dots&0&I_l
	\end{pmatrix}.
	\end{aligned}
	\end{equation}
\end{small}
\begin{theorem}\label{t2.2}
	For $L,~Q,~M$ defined in \eqref{G38} and $\theta(u)$, $u$, $w$, $F(w),~T(w)$  defined in Example \ref{E1}, it holds that 
	\begin{itemize}
		\item[\textnormal{({1})}]     Algorithm \eqref{G36}-\eqref{G37} satisfies  \eqref{V3}-\eqref{V4}.
		\item[\textnormal{({2})}]   $H$ and $G$ satisfying \eqref{V5}-\eqref{V6} are positive definite if $\alpha\in(0,1)$.
	\end{itemize}
\end{theorem}
\begin{proof}
Proof of \textnormal{({1})}.
For $i=1,2,\dots,m$, 	the optimality condition of the $x_i$-subproblem is given by
\begin{equation*}
	\begin{aligned}
0&\in	\partial f_i(\widetilde x_i^k)-A_i^T\lambda^k+\beta A_i^T\sum_{j=1}^{i}A_j(\widetilde{x}_j^{k}-x_j^k)\\&=
	\partial f_i(\widetilde x_i^k)-	A_i^T\widetilde{\lambda}^k+\beta A_i^T\sum_{j=1}^{i}A_j(\widetilde{x}_j^{k}-x_j^k)+A_i^T(\widetilde{\lambda}^k-\lambda^k).
	\end{aligned}
\end{equation*}
The optimality condition of the $\lambda$-subproblem is given by
\begin{equation*}
0=	(\sum_{j=1}^{m}A_j\widetilde{x}_j^k-b)+\frac{1}{\beta}(\widetilde{\lambda}^k-\lambda^k).
\end{equation*}
Combining the above two relations, we obtain
\begin{footnotesize}
	\begin{equation*}
	\begin{aligned}
	&0\in 
	\begin{pmatrix}
	\partial f_1(\widetilde x_1^k)\\
	\partial f_2(\widetilde x_2^k)\\
	\vdots\\
	\partial f_m(\widetilde x_m^k)\\
	0
	\end{pmatrix}+\begin{pmatrix}
	-A_1^T\widetilde \lambda^k\\
	-A_2^T\widetilde\lambda^k\\
	\vdots\\
	-A_m^T\widetilde\lambda^k\\
	\sum_{j=1}^{m}A_j\widetilde{x}_j^k-b
	\end{pmatrix}+\begin{pmatrix}
\beta A_1^TA_1(\widetilde{x}_1^{k}-x_1^k)+A_1^T(\widetilde{\lambda}^k-\lambda^k)\\
	\beta A_2^T\sum_{j=1}^{2}A_j(\widetilde{x}_j^{k}-x_j^k)+A_2^T(\widetilde{\lambda}^k-\lambda^k)\\
	\vdots\\
	\beta A_m^T\sum_{j=1}^{m}A_j(\widetilde{x}_j^{k}-x_j^k)+A_m^T(\widetilde{\lambda}^k-\lambda^k)\\
\frac{1}{\beta}(	\widetilde{\lambda}^k-\lambda^k
)	\end{pmatrix},
	\end{aligned}
	\end{equation*}
\end{footnotesize}
which is equivalent to
\begin{footnotesize}
\begin{equation*}
\begin{aligned}
&0\in 
\begin{pmatrix}
\partial f_1(\widetilde x_1^k)\\
\partial f_2(\widetilde x_2^k)\\
\vdots\\
\partial f_m(\widetilde x_m^k)\\
0
\end{pmatrix}+\begin{pmatrix}
-A_1^T\widetilde \lambda^k\\
-A_2^T\widetilde\lambda^k\\
\vdots\\
-A_m^T\widetilde\lambda^k\\
\sum_{j=1}^{m}A_j\widetilde{x}_j^k-b
\end{pmatrix}+L^T\begin{pmatrix}
I_l&0&\dots&0&I_l\\
I_l&I_l&\ddots&\vdots&I_l\\
\vdots&\vdots&\ddots&0&\vdots\\
I_l&I_l&\dots&I_l&I_l\\
0&0&\dots&0&I_l
\end{pmatrix}L\begin{pmatrix}
\widetilde{x}_1^k-x_1^k\\
\widetilde{x}_2^k-x_2^k\\
\vdots\\
\widetilde{x}_m^k-x_m^k\\
	\widetilde{\lambda}^k-\lambda^k
\end{pmatrix}.
\end{aligned}
\end{equation*}
\end{footnotesize}
Then the prediction step  holds.  The correction step is easy to verified.

\noindent Proof of \textnormal{({2})}. We can refer to \cite[Lemma 7.1, 7.2]{he2021extensions}.
\end{proof}

\subsection{{Algorithm satisfying \eqref{V3}-\eqref{V4} for solving \eqref{P2}}}\label{s2.4}
The following algorithm is presented by He et al. \cite{doi:10.1137/21M1453463} for solving \eqref{P2}:
\begin{framed}
	\noindent{\bf[Prediction step.]} With given $\beta>0$ and  $(x^k,y^k)$, find $(\widetilde{x}^k,\widetilde{y}^k)$ by
	\begin{equation}\label{G9}
	\begin{cases}
	\widetilde{x}^k=\arg\min\limits_{x}\left\{
	\Phi(x,y^k)+\frac{r}{2}\|x-x^k\|^2
	\right\},\\
	\widetilde{y}^k=\arg\max\limits_{y}\left\{
	\Phi([\widetilde{x}^k+\alpha(\widetilde{x}^k-x^k)],y)-\frac{s}{2}\|y-y^k\|^2
	\right\}.
	\end{cases}
	\end{equation}
	
	\noindent{\bf [Correction step.]} Update $(x^{k+1},y^{k+1})$ by
	\begin{equation}\label{G10}
	\begin{pmatrix}
	x^{k+1}\\y^{k+1}
	\end{pmatrix}=	\begin{pmatrix}
	x^{k}\\y^{k}
	\end{pmatrix}-\begin{pmatrix}
	I_n&0\\
	-(1-\alpha)\frac{1}{s}A&I_m
	\end{pmatrix}	\begin{pmatrix}
	x^{k}-\widetilde{x}^k\\y^{k}-\widetilde{y}^k
	\end{pmatrix}.
	\end{equation}
\end{framed}

For convenience, we define
\begin{equation}\label{G41}
L=\begin{pmatrix}
I_n&0\\0&I_m
\end{pmatrix},~~
Q=\begin{pmatrix}
rI_n&A^T\\
\alpha A&sI_m
\end{pmatrix} ~{\rm and} ~M=\begin{pmatrix}
I_n&0\\
-(1-\alpha)\frac{1}{s}A&I_m
\end{pmatrix}.
\end{equation}

\begin{theorem}\label{t2.3}
	For $L,~Q,~M$ defined in \eqref{G41} and $\theta(u), ~u,~w$, $F(w),~T(w)$ as defined  in Example \ref{E2}, it holds that 
	\begin{itemize}
		\item[\textnormal{({1})}]     Algorithm \eqref{G9}-\eqref{G10} satisfies  \eqref{V3}-\eqref{V4}.
		\item[\textnormal{({2})}]   $H$ and $G$ satisfying \eqref{V5}-\eqref{V6} are positive definite if $$
		rs>(1-\alpha+\alpha^2)\rho(A^TA),~\alpha\in[0,1].
$$
	\end{itemize}
\end{theorem}
\begin{proof}
	Proof of (1).
	The optimality condition of the $x$-subproblem reads as:
	\begin{equation*}
	\begin{aligned}
	0&\in\partial f(\widetilde{x}^k)
	-A^T{y}^k+r(\widetilde{x}^k-x^k)\\&=
	\partial f(\widetilde{x}^k)
	-A^T\widetilde{y}^k+r(\widetilde{x}^k-x^k)+A^T(\widetilde{y}^k-y^k).
	\end{aligned}
	\end{equation*}
	The optimality condition of the $y$-subproblem reads as:
	\begin{equation*}
	0\in\partial g(\widetilde{y}^k)+
	A[\widetilde{x}^k+\alpha(\widetilde{x}^k-x^k)]+s(\widetilde{y}^k-y^k).
	\end{equation*}
	Combining  the above two relations together yields that
	\begin{equation}\nonumber
	\begin{aligned}
	&0\in\begin{pmatrix}
	\partial f(\widetilde{x}^k)\\\partial g(\widetilde{y}^k)
	\end{pmatrix}+\begin{pmatrix}
	-A^T\widetilde{y}^k\\A\widetilde{x}^k
	\end{pmatrix}+
	\begin{pmatrix}
	r(\widetilde{x}^k-x^k)+A^T(\widetilde{y}^k-y^k)\\
	\alpha A(\widetilde{x}^k-x^k)+s(\widetilde{y}^k-y^k)
	\end{pmatrix}
	.
	\end{aligned}
	\end{equation}
	This is equivalent to:
	\begin{equation}\nonumber
	\begin{aligned}
	&	0\in\begin{pmatrix}
	\partial f(\widetilde{x}^k)\\\partial g(\widetilde{y}^k)
	\end{pmatrix}+\begin{pmatrix}
	-A^T\widetilde{y}^k\\A\widetilde{x}^k
	\end{pmatrix}+\begin{pmatrix}
	rI_n&A^T\\
	\alpha A&sI_m
	\end{pmatrix}\begin{pmatrix}
	\widetilde{x}^k-x^k\\
	\widetilde{y}^k-y^k
	\end{pmatrix}.
	\end{aligned}
	\end{equation}
	Then we obtain the prediction step  \eqref{G16} satisfy \eqref{G27} with  $L$  and  $Q$ defined in \eqref{G41}.   The correction step is easy to verified.
	
	\noindent Proof of (2). We can refer to \cite[Proposition 4.1]{doi:10.1137/21M1453463}.
\end{proof}
\begin{remark}
When $\alpha=1$, the algorithm \eqref{G9}-\eqref{G10} simplifies to the CP algorithm introduced in \cite{chambolle2011first}. Convergence is established under the condition $rs>\rho(A^TA)$. By setting $\alpha=\frac{1}{2}$ in algorithm \eqref{G9}-\eqref{G10}, the convergence condition is relaxed to $rs>0.75\rho(A^TA)$. This extension broadens the permissible step size compared to the original CP algorithm.
\end{remark}

\section{Faster  prediction-correction framework}\label{S2.2}
We first present the following new prediction-correction framework with $\theta(u)$, $u$, $w$, $F(w)$ and $T(w)$  defined in Example \ref{E1} or \ref{E2},  and then establish the convergence rates.

\begin{framed}
	\noindent{\bf[Prediction step.]} With  given  $w^k$ and $\breve w^{k-1}$, find $\breve{w}^k$  such that
	\begin{equation}\label{G27}\tag{FPS}
	0\in (T-F)(\breve w^k)+F(\widetilde{w}^k)+L^TQL(\widetilde{w}^k-w^k),
	\end{equation}
	where $Q^T +Q\succ 0$ (noting that $Q$ is not necessarily symmetric), $L$ is a matrix  and 
	\begin{equation}\label{G14}
	\widetilde{w}^k=\frac{1}{\tau^k}\breve{w}^{k}-\frac{1-\tau^k}{\tau^k}\breve{w}^{k-1},
	\end{equation}
	and the sequence $\{\tau^k\}$ satisfies the following equality:
	\begin{equation}
	1/\tau^{k-1}=(1-\tau^k)/\tau^k,~ \tau^{-1}\in(0,1) \label{G13}.\tag{Y}
	\end{equation}
	\noindent{\bf [Correction step.]} Update $v^{k+1}=Lw^{k+1}$ by
	\begin{equation}\label{G28}\tag{FCS}
	v^{k+1}=v^k-M(v^k-\widetilde{v}^k),
	\end{equation}where $ \widetilde v^k=L\widetilde w^k$. 
\end{framed}

According to \eqref{G27},  there exists  $\breve g^k\in (T-F)( \breve w^k)$ such that 
\begin{equation}
\begin{aligned}
&(w-\breve w^k)^T(\breve g^k+F(\widetilde w^k)+L^TQL(\widetilde w^k-w^k)=0,~\forall w
\\\Longrightarrow&
\theta(u)-\theta(\breve{u}^k)+(w-\breve{w}^k)^TF(\widetilde{w}^k)\geq(w-\breve{w}^k)^TL^TQL(w^k-\widetilde{w}^k),~\forall w,
\end{aligned}
\end{equation}where  the  inequality using convexity of $\theta$. Hence prediction-correction framework \eqref{G27}-\eqref{G28} infers
the following framework (in variational form).

\begin{framed}
	\noindent{\bf[Prediction step.]} With  given  $w^k$ and $\breve  w^{k-1}$, find $\breve{w}^k$  such that
	\begin{equation}\label{G15}
	\theta(u)-\theta(\breve{u}^k)+(w-\breve{w}^k)^TF(\widetilde{w}^k)\geq(w-\breve{w}^k)^TL^TQL(w^k-\widetilde{w}^k), ~\forall w,
	\end{equation}
	where $Q^T +Q\succ 0$ (noting that $Q$ is not necessarily symmetric), $L$ is a matrix and $\widetilde{w}^k$ defined in \eqref{G14} and $\tau^{k}$ satisfying \eqref{G13}.
	
	\noindent{\bf [Correction step.]} Update $v^{k+1}=Lw^{k+1}$ by
	\begin{equation}
	v^{k+1}=v^k-M(v^k-\widetilde{v}^k),
	\end{equation}where $ \widetilde v^k=L\widetilde w^k$. 
\end{framed}


The following property of the sequence $\{\tau^k\}$ satisfying \eqref{G13} is trivial to verify and hence omitted.
\begin{lemma} \label{L5}
	Let $\{\tau^k\}$ satisfy \eqref{G13}. Then $\tau^k=\frac{1}{\tau^{-1}+k+1}$.
\end{lemma}

\subsection{$O(1/t)$ non-ergodic convergence rate}
We establish $O(1/t)$ non-ergodic convergence rate of the primal dual gap  for the faster prediction-correction framework \eqref{G27}-\eqref{G28} under the conditions \eqref{V5}-\eqref{V6}.
\begin{lemma} \label{L1}
	For the faster prediction-correction framework \eqref{G27}-\eqref{G28} under \eqref{V5}-\eqref{V6}, we have
	\begin{equation}\label{V36}
	\begin{aligned}
	&\frac{1}{\tau^k}[\theta(u)-\theta(\breve{u}^k)]-\frac{1}{\tau^{k-1}}[\theta(u)-\theta(\breve{u}^{k-1})]+(w-\widetilde{w}^k)^TF(\widetilde{w}^k)
	\\\geq&\frac{1}{2}\left(\|v^{k+1}-v\|^2_{H}-\|v^k-v\|^2_{H}+\|v^k-\widetilde{v}^k\|^2_{G}\right)
	,~\forall w,
	\end{aligned}
	\end{equation}where $v=Lw$.
\end{lemma}
\begin{proof}
	Prediction step \eqref{G27} means prediction step \eqref{G15} holds. 
	Multiplying both sides of \eqref{G15} by $(1-\tau^k)/\tau^k$ with $u=\breve{u}^{k-1}$ and $w=\breve{w}^{k-1}$,  and then adding it to \eqref{G15} yields that
	\begin{equation}\label{G1}
	\begin{aligned}
	&\frac{1}{\tau^k}[\theta(u)-\theta(\breve{u}^k)]-\frac{1-\tau^k}{\tau^{k}}[\theta(u)-\theta(\breve{u}^{k-1})]+(w-\widetilde{w}^k)^TF(\widetilde{w}^k)\\\overset{\eqref{G13}}{=}
	&\frac{1}{\tau^k}[\theta(u)-\theta(\breve{u}^k)]-\frac{1}{\tau^{k-1}}[\theta(u)-\theta(\breve{u}^{k-1})]+(w-\widetilde{w}^k)^TF(\widetilde{w}^k)
	\\\geq&(w-\widetilde{w}^k)^TL^TQL(w^k-\widetilde{w}^k)
	=(v-\widetilde{v}^k)^TQ(v^k-\widetilde{v}^k).
	\end{aligned}
	\end{equation}
	On the other hand, we can verify that
	\begin{eqnarray}
	\begin{aligned}
	&(v-\widetilde{v}^k)^TQ(v^k-\widetilde{v}^k)\overset{\eqref{G28}}{=}(v-\widetilde{v}^k)^TH(v^k-{v}^{k+1})\label{G3}\\=&\frac{1}{2}\left(\|v^{k+1}-v\|^2_H-\|v^{k}-v\|^2_H+\|v^k-\widetilde{v}^k\|_H^2-	\|v^{k+1}-\widetilde{v}^k\|_H^2\right)
	\end{aligned}
	\end{eqnarray}and 
	\begin{eqnarray}
	&&	\|v^k-\widetilde{v}^k\|_H^2-	\|v^{k+1}-\widetilde{v}^k\|_H^2\nonumber\\&=&	\|v^k-\widetilde{v}^k\|_H^2-	\|(v^{k}-\widetilde{v}^k)-(v^k-{v}^{k+1})\|_H^2\nonumber\\&\overset{\eqref{G28}}{=}&\|v^k-\widetilde{v}^k\|_H^2-	\|(v^{k}-\widetilde{v}^k)-M(v^k-\widetilde{v}^{k})\|_H^2\label{G2}\\&=&(v^k-\widetilde{v}^k)^T(2HM-M^THM)(v^k-\widetilde{v}^k)\nonumber\\&=&(v^k-\widetilde{v}^k)^T(Q^T+Q-M^THM)(v^k-\widetilde{v}^k)\nonumber\\&=&\|v^k-\widetilde{v}^k\|_G^2. \nonumber\end{eqnarray}
	Combining \eqref{G1}, \eqref{G3} and \eqref{G2}  completes the proof.
\end{proof}

\begin{lemma} \label{L2}
	For the faster prediction-correction framework \eqref{G27}-\eqref{G28} under \eqref{V5}-\eqref{V6}, we have
	\begin{equation}\label{G43}
	\begin{aligned}
	&\frac{1}{\tau^k}[\theta(u)-\theta(\breve{u}^k)+(w-\breve{w}^k)^TF({w})]\\&-\frac{1}{\tau^{k-1}}[\theta(u)-\theta(\breve{u}^{k-1})+(w-\breve{w}^{k-1})^TF({w})]
	\\\geq&\frac{1}{2}\left(\|v^{k+1}-v\|^2_{H}-\|v^k-v\|^2_{H}+\|v^k-\widetilde{v}^k\|^2_{G}\right)
	,~\forall w,
	\end{aligned}
	\end{equation}where $v=Lw$.
\end{lemma}
\begin{proof}
	Based on the definition of $F(w)$  in Example \ref{E1} and \ref{E2}, we have
	\begin{equation*}
	(w-w')^T(F(w)-F(w'))=0,~\forall w,w'.
	\end{equation*}
	Then  we obtain
	\begin{eqnarray}\label{G19}
	&&(w-\widetilde{w}^k)^TF(\widetilde{w}^k)=(w-\widetilde{w}^k)^TF({w})\nonumber\\&=&
	\frac{1}{\tau^k}(w-\breve{w}^k)^TF({w})-\frac{1-\tau^k}{\tau^k}
	(w-\breve{w}^{k-1})^TF({w})\\&\overset{\eqref{G13}}{=}&
	\frac{1}{\tau^k}(w-\breve{w}^k)^TF({w})-\frac{1}{\tau^{k-1}}
	(w-\breve{w}^{k-1})^TF({w}).\nonumber
	\end{eqnarray}
	Combining  Lemma \ref{L1} we obtain the conclusion.
\end{proof}

Then we obtain the following conclusion immediately. 
\begin{theorem}[$O(1/t)$ non-ergodic convergence rate]\label{TNE}
	Let $\{\breve{w}^t\}$ be generated by \eqref{G27}-\eqref{G28} under \eqref{V5}-\eqref{V6}. Then  we have
	\begin{equation*}
	\begin{aligned}
	&\theta(\breve{u}^{t})-\theta(u)-	(w-\breve{w}^t)^TF({w})\\\leq& 
\frac{1}{\tau^{-1}+t+1}\left\{
\frac{1}{\tau^0}	[\theta(\breve{u}^{0})-\theta(u)-	(w-\breve{w}^0)^TF({w})]+\frac{1}{2}\|v^{1}-v\|^2_{H}
\right\},~t=1,2,\dots,~\forall w,
	\end{aligned}
	\end{equation*}where $v=Lw$.
\end{theorem}
\begin{proof}
	Adding Lemma \ref{L2} from $k=1$ to $k=t$, then  
		\begin{equation*}
		\begin{aligned}
	&\frac{1}{\tau^t}	[\theta(\breve{u}^{t})-\theta(u)-	(w-\breve{w}^t)^TF({w})]+\frac{1}{2}\|v^{t+1}-v\|^2_{H}
	\\	\leq&
	\frac{1}{\tau^0}	[\theta(\breve{u}^{0})-\theta(u)-	(w-\breve{w}^0)^TF({w})]+\frac{1}{2}\|v^{1}-v\|^2_{H},~t=1,2,\dots,~\forall w.
		\end{aligned}
		\end{equation*}
Based on Lemma \ref{L5}, we obtain the conclusion.	
\end{proof}

\subsection{$O(1/t^2)$  convergence rate in the pointwise sense}
We establish $O(1/t^2)$  convergence rate in  pointwise sense for the faster prediction-correction framework \eqref{G27}-\eqref{G28} under the conditions \eqref{V5}-\eqref{V6}.

\begin{theorem}[$O(1/t^2)$  in the pointwise sense]\label{T3}
	For the faster prediction-correction framework \eqref{G27}-\eqref{G28} under \eqref{V5}-\eqref{V6}, we have\begin{equation*}
	\|M(\breve{v}^{t}-\breve{v}^{t-1})\|^2_H\leq O(1/t^2),~t=1,2,\dots.
	\end{equation*}
\end{theorem}
\begin{proof}Suppose $w^*$ is a saddle point of \eqref{P3} or \eqref{P2} and $v^*=Lw^*$.
	Based on the correction step \eqref{G28}, we have
	\begin{equation}\label{G4}
	\begin{aligned}
	\|v^{k+1}-v^*\|^2_{H}=&\|(I-M)(v^k-v^*)+M(\widetilde{v}^k-v^*)\|^2_H\\=&\underbrace{\|(I-M)(v^k-v^*)\|^2_H}_{:=A^k}+\underbrace{\|M(\widetilde{v}^k-v^*)\|^2_H}_{:=B^k}\\&+2\underbrace{(v^k-v^*)^T(I-M)^THM(\widetilde{v}^k-v^*)}_{:=C^k}.
	\end{aligned}
	\end{equation}
	According to the definition of $\widetilde{v}^k$, we obtain
	\begin{equation}\label{G5}
	\begin{aligned}
	&\|M(\widetilde{v}^{k}-v^*)\|^2_H+\frac{1-\tau^k}{\tau^k}\underbrace{\|M(\breve{v}^{k-1}-v^*)\|^2_H}_{:=D^{k-1}}
	\\=&\frac{1-\tau^k}{(\tau^k)^2}\underbrace{\|M(\breve{v}^{k}-\breve{v}^{k-1})\|^2_H}_{:=E^k}+\frac{1}{\tau^k}\underbrace{\|M(\breve{v}^{k}-v^*)\|^2_H}_{D^k}.
	\end{aligned}
	\end{equation}
	Setting  $u=u^*$, $w=w^*$ and $v=v^*$ in Lemma \ref{L2},  it follows from \eqref{G4} and \eqref{G5} that
	\begin{small}
	\begin{equation}\label{G6}
	\begin{aligned}
	&\|v^{k+1}-v^*\|^2_{H}-\|v^k-v^*\|^2_{H}=A^k-A^{k-1}+B^k-B^{k-1}+2(C^k-C^{k-1})\\=&A^k-A^{k-1}+2(C^k-C^{k-1})+
	\frac{1-\tau^k}{(\tau^k)^2}E^k-\frac{1-\tau^{k-1}}{(\tau^{k-1})^2}E^{k-1}\\&+\frac{1}{\tau^k}(D^k-D^{k-1})-\frac{1}{\tau^{k-1}}(D^{k-1}-D^{k-2})+D^{k-1}-D^{k-2}\\\leq&
	-\frac{2}{\tau^t}S^t+\frac{2}{\tau^{k-1}}S^{k-1},	
	\end{aligned}
	\end{equation}
	\end{small}
	where
	$S^t=	\theta(\breve{u}^{t})-\theta(u^*)+	(\breve{w}^t-w^*)^TF({w^*})\geq0.$ For each $t=1,2,\dots,$
	summing up both sides of \eqref{G6} from $k=1$ to $t$ yields that
	\begin{eqnarray}
	&&\sum_{k=1}^{t}\left(
	\|v^{k+1}-v^*\|^2_{H}-\|v^k-v^*\|^2_{H}
	\right)\nonumber\\&=&A^t-A^0+2(C^t-C^{0})+
	\frac{1-\tau^t}{(\tau^t)^2}E^t-\frac{1-\tau^{0}}{(\tau^{0})^2}E^{0}\nonumber\\&&+\frac{1}{\tau^t}(D^t-D^{t-1})-\frac{1}{\tau^{0}}(D^{0}-D^{-1})+D^{t-1}-D^{-1}\label{G21}\\&\overset{}{\leq}&-\frac{2}{\tau^t}S^t+\frac{2}{\tau^0}S^0,\nonumber
	\end{eqnarray}
	According to \eqref{G6}, $v^k$ is bounded. Therefore, the correction step \eqref{G28} implies that $\widetilde{v}^k$ is also bounded. Then $C^t$ is bounded by its definition. It follows from \eqref{G21} that there is a positive bound $N_1<\infty$ such that
	\begin{equation}\label{G7}
	\frac{2}{\tau^t}S^t+A^t+	\frac{1-\tau^t}{(\tau^t)^2}E^t+\frac{1}{\tau^t}(D^t-D^{t-1})+
	D^{t-1}<N_1.
	\end{equation}
	Since $S^t,~A^t$, $N^t$ and $E^t$ are all nonnegative, it  implies from \eqref{G7} that
	\begin{equation*}
	\frac{1}{\tau^t}(D^t-D^{t-1})+
	D^{t-1}\overset{\eqref{G13}}{=}\frac{1}{\tau^t}D^t-\frac{1}{\tau^{t-1}}D^{t-1}<N_1.
	\end{equation*}
	Then it holds that
	\begin{equation*}
	\frac{1}{\tau^t}	D^t<tN_1+\frac{1}{\tau^0}D^0.
	\end{equation*}
	Therefore, we have
	\begin{equation*}
	D^t\leq
	\tau^t tN_1+\frac{\tau^t}{\tau^0}D^0<\infty,~t\rightarrow\infty,
	\end{equation*}
	that is, $D^t$ is bounded, i.e., $D^t\leq N_2$ for some $0<N_2<\infty.$ According to Cauchy-Schwartz inequality, we have
	\begin{equation}\label{G8}
	\begin{aligned}
	D^{t-1}-D^{t}=&-(M(\breve{v}^{t}-v^*)+M(\breve{v}^{t-1}-v^*) )^THM(
	\breve{v}^{t}-\breve{v}^{t-1})\\\leq&
	\|M(\breve{v}^{t}-v^*)+M(\breve{v}^{t-1}-v^*)\|_H\|M(
	\breve{v}^{t}-\breve{v}^{t-1})\|_H\\\leq&
	2\sqrt{N_2}\sqrt{E^t}.
	\end{aligned}
	\end{equation}
	Combining 	\eqref{G7} and \eqref{G8} implies that
	\begin{equation}
	\frac{1-\tau^0}{(\tau^t)^2}E^t\leq
	\frac{1-\tau^t}{(\tau^t)^2}E^t\leq N_1+\frac{2}{\tau^t}\sqrt{N_2}\sqrt{E^t}.\label{G78}
	\end{equation}
	Let $h^t:=\sup_{t\geq0} \sqrt{E^t}/\tau^t$. It follows from \eqref{G78} that
	$$(1-\tau^0)h^t\leq\frac{N_1}{h^t}+2\sqrt{N_2}<\infty,$$
	i.e., $E^t\leq O((\tau^t)^2)$ for $t\rightarrow\infty$.  
	By Lemma \ref{L5}, we complete the proof.
\end{proof}

\section{Applications}\label{S4}


A simple way to design algorithms that satisfy our faster prediction-correction framework \eqref{G27}-\eqref{G28} is to convert algorithms satisfying \eqref{V3}-\eqref{V4} to faster versions that satisfy \eqref{G27}-\eqref{G28}.
In particular, we can construct faster algorithms to solve special cases of \eqref{P3} and \eqref{P2} based on the algorithms that satisfy \eqref{V3}-\eqref{V4}, such as ADMM, the linear ADMM \cite{yang2013linearized}, the strictly contractive Peaceman-Rachford splitting method \cite{he2014strictly,he2016convergence}, the Jacobian ALM \cite{doi:10.1137/130922793,xu2021parallel}, the ADMM with a substitution \cite{he2012alternating,he2017convergence}, ADMM-type algorithm \cite{he2021extensions}, CP-type algorithm \cite{doi:10.1137/21M1453463} with a larger step size and so on. In this section, we exemplify three such faster algorithms.

  \subsection{Faster algorithm satisfying \eqref{G27}-\eqref{G28} for solving \eqref{P3} with $m=2$ and equality constraints}\label{s4.3}
  We design a faster algorithm satisfying \eqref{G27}-\eqref{G28} for solving  \eqref{P3} $m=2$.

  We consider the following algorithm 
  for solving \eqref{G23}:
  \begin{framed}
  	\noindent{\bf[Prediction step.]} For given $P\succ 0$, $\beta>0$, $(\breve{x}^{k-1}_1,\breve{x}_2^{k-1},\breve{\lambda}^{k-1})$ and $({x}^{k}_1,{x}_2^{k},{\lambda}^{k})$,  find $(\breve{x}^{k}_1,\breve{x}_2^{k},\breve{\lambda}^{k})$ by
  	\begin{equation}\label{G29}
  	\begin{cases}
  	\breve{x}_1^k\in\arg\min\limits_{x_1}\left\{
  	f_1(x_1)-x_1^TA_1^T\lambda^k+\frac{\beta\tau^k}{2}\|A_1(\frac{1}{\tau^k}x_1-\frac{1-\tau^k}{\tau^k}\breve{x}_1^{k-1})+A_2x_2^k-b\|^2\right.\\\left.~~~~~~~~~~~~~~~~~+\frac{\tau^k}{2}\|\frac{1}{\tau^k}x_1-\frac{1-\tau^k}{\tau^k}\breve{x}_1^{k-1}- x_1^{k}\|^2_P
  	\right\},\\
  	\breve{x}_2^k\in\arg\min\limits_{x_2}\left\{
  	f_2(x_2)-x_2^TA_2^T[
  	\lambda^k-r\beta(A_1\widetilde{x}_1^{k}+A_2x_2^k-b)
  	]\right.\\\left.~~~~~~~~~~~~~~~~~+\frac{\beta\tau^k}{2}\|A_1\widetilde{x}_1^{k}+A_2(\frac{1}{\tau^k}x_2-\frac{1-\tau^k}{\tau^k}\breve{x}_2^{k-1})-b\|^2
  	\right\},\\
  	\breve{\lambda}^k=\arg\max\limits_{\lambda}\left\{-
  	\frac{\tau^k}{2}\|(\frac{1}{\tau^k}\lambda-\frac{1-\tau^k}{\tau^k}\breve{\lambda}^{k-1})-[\lambda^k-\beta(A_1\widetilde{x}_1^{k}+A_2x_2^k-b)]\|^2
  	\right\},
  	\end{cases}
  	\end{equation}where $\tau^k$ satisfy \eqref{G13}.
  	
  	\noindent{\bf [Correction step.]} Update $(x_1^{k+1}, x_2^{k+1},\lambda^{k+1})$ by
  	\begin{equation}\label{G30}
  	\begin{pmatrix}
  x_1^{k+1}\\	x_2^{k+1}\\\lambda^{k+1}
  	\end{pmatrix}=\begin{pmatrix}
  x_1^{k}	\\x_2^{k}\\\lambda^{k}
  	\end{pmatrix}-\begin{pmatrix}
  I_{n_1}&0&0\\0&	I_{n_2}&0\\
  0&	-s\beta A_2&(r+s)I_l
  	\end{pmatrix}\begin{pmatrix}
  x_1^{k}-\widetilde{x}_1^k\\	x_2^{k}-\widetilde{x}_2^k\\\lambda^{k}-\widetilde{\lambda}^k
  	\end{pmatrix},
  	\end{equation}where $$
  	\begin{pmatrix}
  	\widetilde x_1^k\\\widetilde x_2^k\\
  	\widetilde{\lambda}^k
  	\end{pmatrix}=\begin{pmatrix}
  	\frac{1}{\tau^k}\breve x_1^k-\frac{1-\tau^k}{\tau^k}\breve x_1^{k-1}\\ \frac{1}{\tau^k}\breve x_2^k-\frac{1-\tau^k}{\tau^k}\breve x_2^{k-1}\\
  	\frac{1}{\tau^k}\breve \lambda^k-\frac{1-\tau^k}{\tau^k}\breve \lambda^{k-1}
  	\end{pmatrix}.
  	$$
  \end{framed}
  
  The following theorem clarifies that Algorithm \eqref{G29}-\eqref{G30} satisfies \eqref{G27}-\eqref{G28} with convergence conditions  \eqref{V5}-\eqref{V6}. Hence  the rates of $O(1/t)$ in the non-ergodic sense of the primal-dual gap  and $O(1/t^2)$ in the pointwise sense can be obtained by Theorem \ref{TNE} and \ref{T3}. 
  \begin{theorem}
  	For $L,~Q,~M$ defined in \eqref{G25} and $\theta(u)$, $u$, $w$, $F(w),~T(w)$  defined in Example \ref{E1} with $m=2$, it holds that 
  	\begin{itemize}
  		\item[\textnormal{({1})}]    Algorithm \eqref{G29}-\eqref{G30} satisfies \eqref{G27}-\eqref{G28}.
  		\item[\textnormal{({2})}]   $H$ and $G$ satisfying \eqref{V5}-\eqref{V6} are positive definite if $A_2$ is  full  column rank,   \begin{equation*}
  		r\in(-1,1),~s\in(0,1)~{\rm and}~r+s>0.
  		\end{equation*}
  	\end{itemize}
  \end{theorem}
  \begin{proof}
  	Proof of (1).
  	The optimality condition of the $\lambda$-subproblem reads as:
  	\begin{equation}\label{G31}
  0=A_1\widetilde{x}_1^{k}+A_2\widetilde{x}_2^k-b-A_2(\widetilde{x}_2^k-x_2^k)+\frac{1}{\beta}(\widetilde{\lambda}^k-\lambda^k).
  	\end{equation}
  	or equivalently, \begin{equation}\label{G49}
  	\widetilde{\lambda}^k=\lambda^k-\beta(A_1\widetilde{x}_1^{k}+A_2x_2^k-b).
  	\end{equation}
  	The optimality condition of the $x_1$-subproblem is given by
  	\begin{equation}\label{G32}
  	\begin{aligned}
  	0&\in\partial f_1(\breve{x}_1^k)-A_1^T{[\lambda^k-\beta(A_1\widetilde{x}_1^{k}+A_2x_2^k-b)]}+P(\widetilde{x}_1^k-x_1^k)\\&=
  \partial	f_1(\breve{x}_1^k)-A_1^T{\widetilde{\lambda}^k}+P(\widetilde{x}_1^k-x_1^k)
  	.	
  	\end{aligned}
  	\end{equation}
  	The optimality condition of the $x_2$-subproblem reads as:
  	\begin{equation}\label{G33}
  	\begin{aligned}
  	0&\in\partial f_2(\breve{x}_2^k)
  	-A_2^T[
  	\lambda^k-r\beta(A_1\widetilde{x}_1^{k}+A_2x_2^k-b)
  	]+\beta A_2^T[ A_1\widetilde{x}_1^{k}+A_2\widetilde{x}_2^k-b]\\&=
  	\partial f_2(\breve{x}_2^k)-A_2^T\widetilde{\lambda}^k+\beta A_2^T A_2(\widetilde{x}_2^k-x_2^k)-rA_2^T(\widetilde{\lambda}^k-\lambda^k).
  	\end{aligned}
  	\end{equation}
 Combining \eqref{G32}, \eqref{G33} and \eqref{G31},
  	we have
  	\begin{equation*}
  	\begin{aligned}
  	&0\in \begin{pmatrix}
  	\partial f_1(\breve{x}_1^k)\\\partial f_2(\breve{x}_2^k)\\0
  	\end{pmatrix}+
  	\begin{pmatrix}
  	-A_1^T\widetilde{\lambda}^k\\-A_2^T\widetilde{\lambda}^k\\A_1\widetilde{x}_1^{k}+A_2\widetilde{x}_2^k-b
  	\end{pmatrix}
  	+\begin{pmatrix}
  	P(\widetilde{x}_1^k-x_1^k)\\
  	\beta A_2^T A_2(\widetilde{x}_2^k-x_2^k)-rA_2^T(\widetilde{\lambda}^k-\lambda^k) \\
  	-A_2(\widetilde{x}_2^k-x_2^k)+\frac{1}{\beta}(\widetilde{\lambda}^k-\lambda^k)
  	\end{pmatrix},
  	\end{aligned}
  	\end{equation*}
  	which is equivalent to$$
  	\begin{aligned}
  	&	0\in \begin{pmatrix}
  	\partial f_1(\breve{x}_1^k)\\\partial f_2(\breve{x}_2^k)\\0
  	\end{pmatrix}+
  	\begin{pmatrix}
  	-A_1^T\widetilde{\lambda}^k\\-A_2^T\widetilde{\lambda}^k\\A_1\widetilde{x}_1^{k}+A_2\widetilde{x}_2^k-b
  	\end{pmatrix}+{\begin{pmatrix}P&0&0\\
  	0&\beta A_2^TA_2&-rA_2^T\\
  	0&-A_2&\frac{1}{\beta}I_l
  	\end{pmatrix}}_{}{\begin{pmatrix}
  	x_1^k-\widetilde{x}_1^k\\
  x_2^k-\widetilde{x}_2^k\\
  \lambda^k-\widetilde{\lambda}^k
  \end{pmatrix}}.
  	\end{aligned}  	
  	$$
  	Then  we obtain the prediction step  \eqref{G29} satisfying \eqref{G27} with $L$ and $Q$ defined in \eqref{G25}.  The correction step is easy to verified. 
  	
  	\noindent Proof of (2). We can refer to Theorem \ref{t2.1} (2).
  \end{proof}
  \begin{remark}
  	According to the correction step \eqref{G30}, it holds that $$
  	x_1^{k+1}= \widetilde{x}_1^k,~~x_2^{k+1}=\widetilde{x}_2^k.
  	$$ and  \begin{eqnarray*} 
  	\lambda^{k+1}&=&\lambda^k+s\beta A_2(x_2^k-\widetilde{x}_2^k)-(r+s)(\lambda^k-\widetilde{\lambda}^k)\\&\overset{\eqref{G49}}{=}&\underbrace{\lambda^k-r\beta(A_1\widetilde{x}_1^k+A_2{x}_2^k-b)}_{:\lambda^{k+\frac{1}{2}}}-s\beta(A_1\widetilde{x}_1^k+A_2\widetilde{x}_2^k-b)\\&=&\lambda^{k+\frac{1}{2}}-s\beta(A_1\widetilde{x}_1^k+A_2\widetilde{x}_2^k-b)
  	\\&=&\lambda^{k+\frac{1}{2}}-s\beta(A_1{x}_1^{k+1}+A_2{x}_2^{k+1}-b)
  	\end{eqnarray*}
  	Then Algorithm \eqref{G29}-\eqref{G30} can be rewritten as:
  	\begin{equation*}
  		\begin{cases}
  		\breve{x}_1^k&\in\arg\min\limits_{x_1}\left\{
  		f_1(x_1)-x_1^TA_1^T\lambda^k\right.\\&\left.+\frac{\beta\tau^k}{2}\|A_1(\frac{1}{\tau^k}x_1-\frac{1-\tau^k}{\tau^k}\breve{x}_1^{k-1})+A_2x_2^k-b\|^2+\frac{\tau^k}{2}\|\frac{1}{\tau^k}x_1-\frac{1-\tau^k}{\tau^k}\breve{x}_1^{k-1}-x_1^k\|^2_P
  		\right\},\\
  			x_1^{k+1}&=\frac{1}{\tau^k}\breve x_1^k-\frac{1-\tau^k}{\tau^k}\breve x_1^{k-1},\\
  		\lambda^{k+\frac{1}{2}}&=\lambda^k-r\beta(A_1\widetilde{x}_1^k+A_2{x}_2^k-b),\\
  		\breve{x}_2^k&\in\arg\min\limits_{x_2}\left\{
  		f_2(x_2)-x_2^TA_2^T
  		\lambda^{k+\frac{1}{2}}
  		\right.\\&\left.+\frac{\beta\tau^k}{2}\|A_1\widetilde{x}_1^{k}+A_2(\frac{1}{\tau^k}x_2-\frac{1-\tau^k}{\tau^k}\breve{x}_2^{k-1})-b\|^2	\right\},\\
  			x_2^{k+1}&=\frac{1}{\tau^k}\breve x_2^k-\frac{1-\tau^k}{\tau^k}\breve x_2^{k-1},\\
  		\lambda^{k+1}&=\lambda^{k+\frac{1}{2}}-s\beta(A_1\widetilde{x}_1^k+A_2\widetilde{x}_2^k-b).
  	
  		\end{cases}
  	\end{equation*}
  	We can find that this algorithm update $\lambda$ twice in one iteration. Numerically,  updating $\lambda$ twice performance better than updating $\lambda$ once. This is the first paper that updates $\lambda$ twice with non-ergodic convergence. 
  \end{remark}
\subsection{Faster algorithm satisfying \eqref{G27}-\eqref{G28} for solving \eqref{P3}}\label{s4.1}
We design a faster algorithm satisfying \eqref{G27}-\eqref{G28} for solving  \eqref{P3}.


We consider the following algorithm  for solving \eqref{P3}:
\begin{framed}
	\noindent{\bf[Prediction step.]} With given $\beta>0$, $(\breve x_1^{k-1},\breve x_2^{k-1},\dots,\breve x_m^{k-1},\breve\lambda^{k-1})$ and $( x_1^{k}, x_2^{k},\dots, x_m^{k},\lambda^k)$, find $(\breve x_1^{k},\breve x_2^{k},\dots,\breve x_m^{k},\breve\lambda^k)$ by
	\begin{small}
		\begin{equation}\label{G34}
		\begin{cases}
		\breve{x}_1^k=\arg\min\limits_{x_1}\left\{
		f_1(x_1)-x_1^TA_1^T\lambda^k+\frac{\beta\tau^k}{2}\|A_1([\frac{1}{\tau^k}x_1-\frac{1-\tau^k}{\tau^k}\breve{x}_1^{k-1}]-x_1^k)\|^2
		\right\},\\
		\breve{x}_2^k=\arg\min\limits_{x_2}\left\{
		f_2(x_2)-x_2^TA_2^T
		\lambda^k\right.\\\left.~~~~~~~~~+\frac{\beta\tau^k}{2}\|A_1(\widetilde{x}_1^{k}-x_1^k)+A_2([\frac{1}{\tau^k}x_2-\frac{1-\tau^k}{\tau^k}\breve{x}_2^{k-1}]-x_2^k)\|^2
		\right\},\\
		~~~~~~~~~\vdots\\
		\breve{x}_i^k=\arg\min\limits_{x_i}\left\{
		f_i(x_i)-x_i^TA_2^T
		\lambda^k\right.\\\left.~~~~~~~~+\frac{\beta\tau^k}{2}\|\sum_{j=1}^{i-1}A_j(\widetilde{x}_j^{k}-x_j^k)+A_i([\frac{1}{\tau^k}x_i-\frac{1-\tau^k}{\tau^k}\breve{x}_i^{k-1}]-x_i^k)\|^2
		\right\},\\
		~~~~~~~~~\vdots\\
		\breve{x}_m^k=\arg \min\limits_{x_m}\left\{
		f_m(x_m)-x_m^TA_2^T
		\lambda^k\right.\\\left.~~~~~~~~+\frac{\beta\tau^k}{2}\|\sum_{j=1}^{m-1}A_j(\widetilde{x}_j^{k}-x_j^k)+A_m([\frac{1}{\tau^k}x_m-\frac{1-\tau^k}{\tau^k}\breve{x}_m^{k-1}]-x_m^k)\|^2
		\right\},\\
		\breve{\lambda}^k=\arg\max\limits_{\lambda}\left\{
		-\lambda^T(\sum_{j=1}^{m}A_j\widetilde{x}_j^k-b)-\frac{\tau^k}{2\beta}\|[\frac{1}{\tau^k}\lambda-\frac{1-\tau^k}{\tau^k}\breve{\lambda}^{k-1}]-\lambda^k\|^2
		\right\},
		\end{cases}
		\end{equation}
	\end{small}where $\tau^k$ satisfy \eqref{G13}.
	
	\noindent{\bf [Correction step.]} Update  $(A_1x_1^{k+1},A_2x_2^{k+1},\dots,A_mx_m^{k+1},\lambda^{k+1})$ by
	\begin{footnotesize}
		\begin{equation}\label{G35}
		\begin{pmatrix}
		\sqrt{\beta}A_1x_1^{k+1}\\	\sqrt{\beta}A_2x_2^{k+1}\\\vdots\\	\sqrt{\beta}A_mx_m^{k+1}\\\frac{1}{	\sqrt{\beta}}\lambda^{k+1}
		\end{pmatrix}=\begin{pmatrix}
		\sqrt{\beta}A_1x_1^{k}\\	\sqrt{\beta}A_2x_2^{k}\\\vdots\\	\sqrt{\beta}A_mx_m^{k}\\\frac{1}{	\sqrt{\beta}}\lambda^{k}
		\end{pmatrix}-\begin{pmatrix}
		\alpha I_l&-\alpha I_l&0&\dots&0\\
		0&\alpha I_l&\ddots&\ddots&\vdots\\
		\vdots&\ddots&\ddots&-\alpha I_l&0\\
		0&\dots&0&\alpha I_l&0\\
		-\alpha I_l&0&\dots&0&I_l
		\end{pmatrix}\begin{pmatrix}
		\sqrt{\beta}( 	A_1x_1^{k}-A_1\widetilde{x}_1^k)\\	\sqrt{\beta}(A_2x_2^{k}-A_2\widetilde{x}_2^k)\\\vdots\\	\sqrt{\beta}(A_mx_m^{k}-A_m\widetilde{x}_m^k)\\\frac{1}{	\sqrt{\beta}}(\lambda^{k}-\widetilde{\lambda}^k)
		\end{pmatrix},
		\end{equation}
	\end{footnotesize}where$$
	\begin{pmatrix}
	\widetilde x_1^k\\\vdots\\\widetilde x_m^k\\
	\widetilde{\lambda}^k
	\end{pmatrix}=\begin{pmatrix}
	\frac{1}{\tau^k}\breve x_1^k-\frac{1-\tau^k}{\tau^k}\breve x_1^{k-1}\\\vdots\\ \frac{1}{\tau^k}\breve x_m^k-\frac{1-\tau^k}{\tau^k}\breve x_m^{k-1}\\
	\frac{1}{\tau^k}\breve \lambda^k-\frac{1-\tau^k}{\tau^k}\breve \lambda^{k-1}
	\end{pmatrix}.
	$$
\end{framed}

The following theorem clarifies that Algorithm \eqref{G34}-\eqref{G35} satisfies \eqref{G27}-\eqref{G28} with convergence conditions  \eqref{V5}-\eqref{V6}. Hence the rates of $O(1/t)$ in the non-ergodic sense of the primal-dual gap  and $O(1/t^2)$ in the pointwise sense can be obtained by Theorem \ref{TNE} and \ref{T3}. 
\begin{theorem}
		For $L,~Q,~M$ defined in \eqref{G38} and $\theta(u)$, $u$, $w$, $F(w),~T(w)$  defined in Example \ref{E1}, it holds that 
		\begin{itemize}
			\item[\textnormal{({1})}]     Algorithm \eqref{G34}-\eqref{G35} satisfies \eqref{G27}-\eqref{G28}.
			\item[\textnormal{({2})}]   $H$ and $G$ satisfying \eqref{V5}-\eqref{V6} are positive definite if $\alpha\in(0,1)$.
		\end{itemize}
\end{theorem}
\begin{proof} Proof of (1).
	For $i=1,2,\dots,m$, 	the optimality condition of the $x_i$-subproblem is given by
	\begin{equation}\label{G40}
	\begin{aligned}
0&\in\partial f_i(\breve x_i^k)
	-A_i^T\lambda^k+\beta A_i^T\sum_{j=1}^{i}A_j(\widetilde{x}_j^{k}-x_j^k)\\&=
	\partial f_i(\breve x_i^k)-A_i^T\widetilde{\lambda}^k+\beta A_i^T\sum_{j=1}^{i}A_j(\widetilde{x}_j^{k}-x_j^k)+A_i^T(\widetilde{\lambda}^k-\lambda^k)
	\end{aligned}
	\end{equation}
	The optimality condition of the $\lambda$-subproblem reads as:
	\begin{equation}\label{G39}
	0=(\sum_{j=1}^{m}A_j\widetilde{x}_j^k-b)+\frac{1}{\beta}(\widetilde{\lambda}^k-\lambda^k).
	\end{equation}
	Combining \eqref{G40} and \eqref{G39}, we obtain 
	$$
	\begin{small}\begin{aligned}
&		0\in \begin{pmatrix}
\partial f_1(\breve{x}_1^k)\\\partial f_2(\breve{x}_2^k)\\\vdots\\\partial f_m(\breve x_m)\\0
\end{pmatrix}+\begin{pmatrix}
-A_1^T\widetilde\lambda^k\\
-A_2^T\widetilde\lambda^k\\
\vdots\\
-A_m^T\widetilde\lambda^k\\
\sum_{j=1}^{m}A_j\widetilde{x}_j^k-b
	\end{pmatrix}+
\begin{pmatrix}
\beta A_1^TA_1(\widetilde{x}_1^{k}-x_1^k)+A_1^T(\widetilde{\lambda}^k-\lambda^k)\\
\beta A_2^T\sum_{j=1}^{2}A_j(\widetilde{x}_j^{k}-x_j^k)+A_2^T(\widetilde{\lambda}^k-\lambda^k)\\
\vdots\\
\beta A_m^T\sum_{j=1}^{m}A_j(\widetilde{x}_j^{k}-x_j^k)+A_m^T(\widetilde{\lambda}^k-\lambda^k)\\
\frac{1}{\beta}(	\widetilde{\lambda}^k-\lambda^k
)	\end{pmatrix},
	\end{aligned}
	\end{small}
	$$
	which is equivalent to:
	$$
	\begin{small}\begin{aligned}
	&		0\in \begin{pmatrix}
	\partial f_1(\breve{x}_1^k)\\\partial f_2(\breve{x}_2^k)\\\vdots\\\partial f_m(\breve x_m)\\0
	\end{pmatrix}+\begin{pmatrix}
	-A_1^T\widetilde\lambda^k\\
	-A_2^T\widetilde\lambda^k\\
	\vdots\\
	-A_m^T\widetilde\lambda^k\\
	\sum_{j=1}^{m}A_j\widetilde{x}_j^k-b
	\end{pmatrix}+L^T{\begin{pmatrix}
	I_l&0&\dots&0&I_l\\
	I_l&I_l&\ddots&\vdots&I_l\\
	\vdots&\vdots&\ddots&0&\vdots\\
	I_l&I_l&\dots&I_l&I_l\\
	0&0&\dots&0&I_l
	\end{pmatrix}} L\begin{pmatrix}
\widetilde{x}_1^k-x_1^k\\
\widetilde{x}_2^k-x_2^k\\
\vdots\\
\widetilde{x}_m^k-x_m^k\\
\widetilde{\lambda}^k-\lambda^k
\end{pmatrix}.
	\end{aligned}
	\end{small}
	$$
	 Then the prediction step  \eqref{G35} satisfies \eqref{G27} with $L$ and $Q$ defined in \eqref{G38}.  The correction step is easy to verified.
	 
	 	\noindent Proof of (2). We can refer to Theorem \ref{t2.2} (2).
\end{proof}

\subsection{Faster algorithm satisfying \eqref{G27}-\eqref{G28} for solving \eqref{P2}}\label{s4.2}
This subsection presents an algorithm satisfying \eqref{G27}-\eqref{G28} for solving  \eqref{P2}.


We consider the following algorithm   for solving \eqref{P2}:
\begin{framed}
	\noindent{\bf[Prediction step.]} With given $r,s>0$, $(\breve{x}^{k-1},\breve{y}^{k-1})$ and $({x}^k,{y}^k)$, find $(\breve{x}^k,\breve{y}^k)$ by
	\begin{equation}\label{G16}
	\begin{cases}
	\breve{x}^k=\arg\min\limits_{x}\left\{
	\Phi(x,y^k)+\frac{r\tau^k}{2}\|(\frac{1}{\tau^k}x-\frac{1-\tau^k}{\tau^k}\breve{x}^{k-1})-x^k\|^2
	\right\},\\
	\breve{y}^k=\arg\max\limits_{y}\left\{
	\Phi([\widetilde{x}^k+\alpha(\widetilde{x}^k-x^k)],y)-\frac{s\tau^k}{2}\| (\frac{1}{\tau^k}y-\frac{1-\tau^k}{\tau^k}\breve{y}^{k-1})-y^k\|^2
	\right\},
	\end{cases}
	\end{equation}where $\tau^k$ satisfy \eqref{G13}. 
	
	\noindent{\bf [Correction step.]} Update $(x^{k+1},y^{k+1})$ by
	\begin{equation}\label{G17}
	\begin{pmatrix}
	x^{k+1}\\y^{k+1}
	\end{pmatrix}=	\begin{pmatrix}
	x^{k}\\y^{k}
	\end{pmatrix}-\begin{pmatrix}
	I_n&0\\
	-(1-\alpha)\frac{1}{s}A&I_m
	\end{pmatrix}	\begin{pmatrix}
	x^{k}-\widetilde{x}^k\\y^{k}-\widetilde{y}^k
	\end{pmatrix},
	\end{equation}where $$
	\begin{pmatrix}
	\widetilde x^k\\\widetilde y^k
	\end{pmatrix}=\begin{pmatrix}
	\frac{1}{\tau^k}\breve x^k-\frac{1-\tau^k}{\tau^k}\breve x^{k-1}\\ \frac{1}{\tau^k}\breve y^k-\frac{1-\tau^k}{\tau^k}\breve y^{k-1}
	\end{pmatrix}.
	$$
\end{framed}

The following theorem clarifies that Algorithm \eqref{G16}-\eqref{G17} satisfies \eqref{G27}-\eqref{G28} with convergence conditions  \eqref{V5}-\eqref{V6}. Hence the rates of  $O(1/t)$ in the non-ergodic sense of the primal-dual gap  and $O(1/t^2)$ in the pointwise sense can be obtained by Theorem \ref{TNE} and \ref{T3}. 
	\begin{theorem}
		For $L,~Q,~M$ defined in \eqref{G38} and $\theta(u), ~u,~w$, $F(w),~T(w)$ as in Example \ref{E2}, it holds that 
		\begin{itemize}
			\item[\textnormal{({1})}]     	Algorithm \eqref{G16}-\eqref{G17} satisfies \eqref{G27}-\eqref{G28}.
			\item[\textnormal{({2})}]   $H$ and $G$ satisfying \eqref{V5}-\eqref{V6} are positive definite if $$
			rs>(1-\alpha+\alpha^2)\rho(A^TA),~\alpha\in[0,1].
			$$
		\end{itemize}
	\end{theorem}
\begin{proof}
	Proof of (1).
	The optimality condition of the $x$-subproblem reads as:
	\begin{eqnarray*}
	0&\in&	\partial f(\breve{x}^k)
		-A^T{y}^k+r(\widetilde{x}^k-x^k)\\&=&
\partial f(\breve{x}^k)
	-A^T\widetilde{y}^k+r(\widetilde{x}^k-x^k)+A^T(\widetilde{y}^k-y^k).
	\end{eqnarray*}
	The optimality condition of the $y$-subproblem reads as:
	\begin{equation*}
	0\in\partial g(\breve{y}^k)+
	A[\widetilde{x}^k+\alpha(\widetilde{x}^k-x^k)]+s(\widetilde{y}^k-y^k).
	\end{equation*}
	Combining  the above two relations together yields that
	\begin{equation}\nonumber
	\begin{aligned}
	&0\in\begin{pmatrix}
	\partial f(\breve{x}^k)\\\partial g(\breve{y}^k)
	\end{pmatrix}+\begin{pmatrix}
		-A^T\widetilde{y}^k\\A\widetilde{x}^k
	\end{pmatrix}+
	\begin{pmatrix}
	r(\widetilde{x}^k-x^k)+A^T(\widetilde{y}^k-y^k)\\
	\alpha A(\widetilde{x}^k-x^k)+s(\widetilde{y}^k-y^k)
	\end{pmatrix}
.
	\end{aligned}
	\end{equation}
	This is equivalent to:
		\begin{equation}\nonumber
		\begin{aligned}
		&	0\in\begin{pmatrix}
		\partial f(\breve{x}^k)\\\partial g(\breve{y}^k)
		\end{pmatrix}+\begin{pmatrix}
		-A^T\widetilde{y}^k\\A\widetilde{x}^k
		\end{pmatrix}+\begin{pmatrix}
		rI_n&A^T\\
		\alpha A&sI_m
		\end{pmatrix}\begin{pmatrix}
	\widetilde{x}^k-x^k\\
	\widetilde{y}^k-y^k
		\end{pmatrix}.
		\end{aligned}
		\end{equation}
	Then we obtain the prediction step  \eqref{G16} satisfies \eqref{G27} with  $L$  and  $Q$ defined in \eqref{G41}.   The correction step is easy to verified.
	
	\noindent Proof of (2). We can refer to Theorem \ref{t2.3} (2).
\end{proof}

\section{Conclusions}
We present a faster prediction-correction framework to build $O(1/t)$  convergence rate in the non-ergodic sense and $O(1/t^2)$  convergence rate in the pointwise sense without any additional assumption. In comparison, He and Yuan's framework achieves an $O(1/t)$ convergence rate in both the ergodic and the pointwise senses. Our framework can provide faster algorithms for solving general convex optimization problems. In particular, we present three  faster algorithms:  ADMM-type algorithm with dual variable updating twice for solving two-block separable convex optimization with equality linear constraints,  multi-block ADMM-type algorithm for solving multi-block separable convex optimization problems with linear equality   constraints and CP-type algorithm for solving min-max problems with larger step sizes ($rs>0.75\rho(A^TA)$). Future works include in-depth understanding our framework, for example, from the view point of second-order differential equations, establishing the weak convergence of the iterative sequence and  the KKT measure.


\bibliographystyle{siamplain}
\bibliography{ref}

\end{document}